\pdfoutput=1

\documentclass[11pt]{article}

\usepackage[letterpaper,left=1.25in,right=1.25in,top=0.9in,bottom=1.1in]{geometry}
\usepackage{amsmath}
\usepackage{amssymb}
\usepackage{amsthm}
\usepackage{xspace}
\usepackage[T1]{fontenc}
\usepackage[utf8]{inputenc}
\usepackage[english]{babel}
\usepackage{mathpazo}
\DeclareMathAlphabet{\mathcal}{OMS}{cmsy}{m}{n}
\usepackage{microtype}
\usepackage{mathtools}
\usepackage{enumitem}
\usepackage{ifthen}
\setlist[enumerate,1]{label={(\arabic*)}}
\setlist[enumerate,2]{label={(\roman*)}}
\usepackage{tikz-cd}
\usepackage[style=alphabetic,maxnames=10,maxalphanames=4,backend=biber,hyperref=true]{biblatex}
\usepackage{ifpdf}
\ifpdf  \usepackage[pdftex,bookmarks=false]{hyperref}
\else   \usepackage[dvips,bookmarks=false]{hyperref}
        \usepackage{breakurl}
\fi
\usepackage{color}
\definecolor{darkgreen}{rgb}{0,0.45,0}
\hypersetup{colorlinks,urlcolor=blue,citecolor=darkgreen,linkcolor=darkgreen,linktocpage}
\usepackage[capitalise]{cleveref}
\usepackage{verbatim}

\usepackage{rotating}
\newcommand{\intperp}{{\begin{sideways}$\!\Vdash$\end{sideways}}}
\newcommand{\iperp}{\mathbin{\intperp}}

\usepackage{tocloft}
\newboolean{end}
\setboolean{end}{false}

\def\noteson{\gdef\note##1{\noindent{\color{blue}[##1]}}}

\noteson

\newtheorem{mainthm}{Theorem}

\newtheorem{thm}{Theorem}[section]
\newtheorem{lem}[thm]{Lemma}
\newtheorem{prop}[thm]{Proposition}
\crefname{prop}{Proposition}{Propositions}

\theoremstyle{definition}
\newtheorem{defn}[thm]{Definition}
\newtheorem{rmk}[thm]{Remark}
\newtheorem{eg}[thm]{Example}

\newtheorem{introquestion}{Question}
\newtheorem{introeg}[introquestion]{Example}

\newcommand{\lra}       {\longrightarrow}
\newcommand{\llra}[1]   {\stackrel{#1}{\lra}}
\newcommand{\hooklongrightarrow}{\lhook\joinrel\longrightarrow}

\mathchardef\usc="2D
\newcommand{\islocal}[1]{\mathsf{is\usc{}local}_{#1}}
\newcommand{\issurj}{\mathsf{isSurj}}

\newcommand{\fib}{\mathsf{fib}}
\newcommand{\blank}{\mathord{\hspace{1pt}\text{--}\hspace{1pt}}}
\newcommand{\degg}{\mathsf{deg}}
\newcommand{\sphere}[1]{\mathbf{S}^{#1}}
\newcommand{\base}{\ensuremath{\mathsf{base}}\xspace}
\newcommand{\lloop}{\ensuremath{\mathsf{loop}}\xspace}
\newcommand{\unit}{\mathbf{1}}

\newcommand{\susp}{\mathsf{\Sigma}}
\newcommand{\N}{\mathbb{N}}
\newcommand{\Q}{\mathbb{Q}}
\newcommand{\Z}{\mathbb{Z}}
\newcommand{\ap}{\mathsf{ap}}

\newcommand{\UU}{\mathcal{U}}
\newcommand{\Prop}{\mathsf{Prop}}

\newcommand{\id}{\mathsf{id}}

\newcommand{\pr}{\mathsf{pr}}

\newcommand{\jdeq}{\equiv}
\newcommand{\defeq}{:\equiv}

\newcommand{\define}[1]{\textbf{\boldmath{#1}}}
\newcommand{\prebot}[1]{\prescript{\bot}{}{#1}}

\makeatletter
\newcommand{\ctsym}{%
  \mathchoice{\mathbin{\raisebox{0.5ex}{$\displaystyle\centerdot$}}}%
             {\mathbin{\raisebox{0.5ex}{$\centerdot$}}}%
             {\mathbin{\raisebox{0.25ex}{$\scriptstyle\,\centerdot\,$}}}%
             {\mathbin{\raisebox{0.1ex}{$\scriptscriptstyle\,\centerdot\,$}}}
  }
\newcommand{\ct}[3][]{
  \@ifnextchar\bgroup
    {#2 \mathbin{\ctsym_{#1}} \ct[#1]{#3}}
    {#2 \mathbin{\ctsym_{#1}} #3}
  }
\def\smsym{\sum}
\newcommand{\@thesum}[1]{\smsym_{#1}}
\newcommand{\sm}[1]{\@ifnextchar\bgroup{\@sm{#1}\sm}{\@sm{#1}}}
\newcommand{\@sm}[1]{%
  \mathchoice{%
    {{\,\@thesum{#1}\;}}}{\@thesum{(#1)}}{\@thesum{(#1)}}{\@thesum{(#1)}}}

\newcommand{\dsm}[1]{{\displaystyle\sm{#1}}}

\newcommand{\@ifnextchar@starorbrace}[2]
  {\@ifnextchar*{#1}{\@ifnextchar\bgroup{#1}{#2}}}
\newcommand{\@theprd}[1]{\prod_{#1}}
\newcommand{\prd}{\@ifnextchar*{\@iprd}{\@prd}}
\newcommand{\@prd}[1]
  {\@ifnextchar@starorbrace
    {\@tprd{#1}\prd}
    {\@tprd{#1}}}
\newcommand{\@tprd}[1]{%
  \mathchoice{%
    {{\,\@theprd{#1}\,}}}{\@theprd{(#1)}}{\@theprd{(#1)}}{\@theprd{(#1)}}}
\makeatother

\usepackage{titlefoot}

\title{Characterizations of modalities and lex modalities}
\author{J. Daniel Christensen \and Egbert Rijke}
\date{October 15, 2021}

\keywords{reflective subuniverse, localization, modality, lex modality,
factorization system, homotopy type theory.}

\amssubj{18N55 (Primary), 03B38, 03B45, 03G30, 18N60, 55P60, 55U35 (Secondary).}

\begin{document}

\maketitle

\begin{abstract}
  A reflective subuniverse in homotopy type theory is an internal
  version of the notion of a localization in topology or in
  the theory of $\infty$-categories.
  Working in homotopy type theory,
  we give new characterizations of the following conditions on
  a reflective subuniverse $L$:
  (1) the associated subuniverse $L'$ of $L$-separated types is a modality;
  (2) $L$ is a modality;
  (3) $L$ is a lex modality; and
  (4) $L$ is a cotopological modality.
  In each case, we give several necessary and sufficient conditions.
  Our characterizations involve various families of maps associated to $L$,
  such as the $L$-\'etale maps, the $L$-equivalences, the $L$-local maps,
  the $L$-connected maps, the unit maps $\eta_X$, and their left and/or
  right orthogonal complements.
  More generally, our main theorem gives an overview of how all of these
  classes related to each other.
  We also give examples that show that all of the inclusions we describe
  between these classes of maps can be strict.
\end{abstract}

\tableofcontents

\setcounter{secnumdepth}{0}
\section{Introduction}

The study of localizations has a long history, beginning with work
of Sullivan~\cite{Sullivan71} and Bousfield~\cite{Bousfield75}
for topological spaces.
This framework has played a fundamental organizing role in algebraic
topology in the intervening decades, influencing and leading to the solution
of many central conjectures in the field.
Bousfield worked with model categories, but more recently the theory
of localizations was adapted to the setting of $\infty$-categories
by Lurie~\cite{lurie}.
In all cases, a localization $L$ associates to an object $X$ a new
object $LX$, with certain properties.

Independently, logicians and philosophers have considered the notion
of modalities in logic, which allow one to qualify statements to
express \emph{possibility}, \emph{necessity}, and other attributes
such as \emph{temporal} modalities.
These modalities are expressed by applying a modal operator $\Diamond$ to
a proposition $P$ to produce a new proposition $\Diamond P$, with
certain properties.
Overviews be found in~\cite{BRV}, \cite{sep-logic-modal} and~\cite{dpgm:modal-tt}.

Homotopy type theory is a formal system which has models
in all $\infty$-toposes \cite{KL,LS,S,dB,dBB}%
\footnote{The initiality and semantics of higher inductive types still need to be fully worked out.}.
As such, it provides a convenient way to prove theorems for all $\infty$-toposes,
with the types interpreted as \emph{objects} of an $\infty$-topos.
At the same time, types are used to encode \emph{propositions} that one
may wish to prove.
It turns out that the notion of a \emph{reflective subuniverse} in homotopy type theory
simultaneously encodes the idea of a localization of an $\infty$-topos
(see below for more information)
and common modalities studied in logic \cite{Corfield}.

Going further, homotopy type theory has been extended with modalities
to allow it to express \emph{cohesion} \cite{Myers19,Shulman18,SS14,Wellen}, 
which can capture continuous and smooth geometry.

Some reflective subuniverses are better than others.
For example, some are $\Sigma$-closed (modalities), some are left exact (lex modalities), etc.
These are key properties that a reflective subuniverse may or may not have.
Moreover, we will see below that each reflective subuniverse $L$ determines
various classes of maps, such as the $L$-equivalences and the $L$-connected maps.
The goal of this paper is to characterize the reflective subuniverses
that have these and other good properties in terms of conditions on the classes
of maps that they determine.

All of our results are stated and proved within the framework of
homotopy type theory.
It follows that similar results hold in any $\infty$-topos.
However, one must take care in interpreting the results.
For example, the interpretation of a reflective subuniverse does
not correspond to a single reflective subcategory, but rather to
a family of reflective subcategories of each slice category,
compatible with pullback (see~\cite[Appendix~A]{RSS} and~\cite{Vergura1}).
As another example, the interpretation of two maps being orthogonal,
which we denote $f \bot g$,
corresponds to the notion of two maps being \emph{internally} orthogonal
in an $\infty$-topos, which is denoted $f \iperp g$ in~\cite{ABFJ}.

\medskip

We next provide some background, and then
state our main results in an omnibus theorem.

\pagebreak

\subsection{Background}

We follow the conventions and notation of~\cite{hottbook}.
For example, we write the type of maps from $X$ to $Y$ as
either $X \to Y$ or $Y^X$.
We fix a univalent universe $\UU$.
Recall from \cite{hottbook} that a \define{reflective subuniverse} $L$ consists of a subuniverse
$\islocal{L} : \UU \to \Prop$,
a function $L : \UU \to \UU$
and a \define{localization} $\eta_X : X \to LX$ for each $X : \UU$.
Being a localization means that $LX$ is $L$-local and $\eta_X$ is
initial among maps whose codomain is local. 
We call $\eta_X$ the \define{unit} map, and when context makes it clear,
we omit the subscript.

For any reflective subuniverse $L$, it is immediate by the universal property of $L$-localizations that the operation $L$ is functorial, i.e., for any map $f:X\to Y$ there is a unique map $Lf:LX\to LY$ such that the square
\begin{equation*}
  \begin{tikzcd}
    X \arrow[r,"\eta"] \arrow[d,swap,"f"] & LX \arrow[d,"Lf"] \\
    Y \arrow[r,swap,"\eta"] & LY
  \end{tikzcd}
\end{equation*}
commutes. This square is called the \define{$L$-naturality square} of $f$. We say that a map $f$ is
\begin{itemize}
\item \define{$L$-local} if its fibres are $L$-local,
\item \define{$L$-\'etale} if its $L$-naturality square is a pullback square,
\item \define{$L$-connected} if its fibres are $L$-connected (i.e., their $L$-localizations are contractible), and
\item an \define{$L$-equivalence} if $Lf$ is an equivalence.
\end{itemize}
The notion of $L$-\'etale map is a variant of the notion of formally \'etale maps from algebraic geometry and is due to Schreiber~\cite[Definition 3.1]{SchreiberSomeThoughts}. Wellen studied the notion of formally \'etale maps extensively in his PhD thesis~\cite{Wellen}.

We recall from \cite{RSS} that a reflective subuniverse is said to be a \define{modality}%
\footnote{More properly, this should be called an ``idempotent monadic modality,''
but as these are the only ones we consider, we use the term ``modality.''}
if for any family $P$ of $L$-local types over an $L$-local type $X$, the dependent sum $\sm{x:X}P(x)$ is again $L$-local.
That is, the $L$-local types are \define{$\Sigma$-closed}.
By~\cite[Theorem~1.32]{RSS}, it is equivalent to require that the unit maps
$\eta_X$ are $L$-connected for each type $X$.
Furthermore, a reflective subuniverse $L$ is said to be \define{lex} (short for \define{left exact}) if the operation $L$ preserves pullbacks.

By a \define{class of maps}, we mean a subtype of $\sm{X, Y : \UU} Y^X$
defined by a predicate valued in $\Prop$.
Recall that a class of maps $\mathcal{L}$ is left orthogonal to a class of maps $\mathcal{R}$, if for every $f:A\to B$ in $\mathcal{L}$ and for every $g:X\to Y$ in $\mathcal{R}$, the square
\begin{equation*}
  \begin{tikzcd}
    X^B \arrow[r,"f^\ast"] \arrow[d,swap,"g_\ast"] & X^A \arrow[d,"g_\ast"] \\
    Y^B \arrow[r,swap,"f^\ast"] & Y^A
  \end{tikzcd}
\end{equation*}
is a pullback.
More generally, if $\mathcal{M}$ is a class of maps, we will write $\mathcal{M}^\bot$ for the class of maps that are right orthogonal to $\mathcal{M}$, and we will write $\prebot{\mathcal{M}}$ for the class of maps that are left orthogonal to $\mathcal{M}$. So we have $\mathcal{R}\subseteq\mathcal{L}^\bot$ and $\mathcal{L}\subseteq\prebot{\mathcal{R}}$, if $\mathcal{L}$ is left orthogonal to $\mathcal{R}$.

An \define{orthogonal factorization system} on $\UU$ is a pair $(\mathcal{L},\mathcal{R})$ of classes of maps in $\UU$
such that $\mathcal{L}=\prebot{\mathcal{R}}$, $\mathcal{R}=\mathcal{L}^\bot$, and every map $h:X\to Y$ factors as a left map followed by a right map. Every modality gives rise to two orthogonal factorization systems:
\begin{itemize}
\item The \define{stable factorization system}, in which the left class $\mathcal{L}$ consists of the $L$-connected maps, and the right class $\mathcal{R}$ consists of the $L$-local maps.  (See~\cite[Section~1]{RSS} and~\cite{ABFJ}.)
\item The \define{reflective factorization system}, in which the left class $\mathcal{L}$ consists of the $L$-equiv\-a\-lences, and the right class $\mathcal{R}$ consists of the $L$-\'etale maps.  (See~\cite[Theorem~7.2]{CR}.)
\end{itemize}
However, both of these pairs of classes of maps fail to be orthogonal factorization systems in the more general case of a reflective subuniverse.
Our main result will give a thorough understanding of how these four classes of maps relate.

We also recall that a type $X$ is said to be \define{$L$-separated} if its identity types are $L$-local. In~\cite{CORS} it was shown that the subuniverse of $L$-separated types is again a reflective subuniverse, which we will call $L'$.
The unit maps for $L'$ are written $\eta' : X \to L'X$,
and are always surjective~\cite[Lemma~2.17]{CORS}.
When $L$ is a modality, $L'$ is also a modality.
One reason for studying the reflective subuniverse $L'$ is that many
properties that hold for a lex reflective subuniverse also hold for
any reflective subuniverse if they are stated using $L$ and $L'$ together;
this observation can be used to prove results about reflective
subuniverses that are not lex~\cite{CORS}.

Finally, we recall from~\cite{RSS} that a reflective subuniverse is said to be \define{accessible} if it can be presented as the subuniverse of $f$-local types, for a family of maps $\{f_i:A_i\to B_i\}_i$ in $\UU$, indexed by a type $I$ in $\UU$. Here, a type $X$ is said to be \define{$f$-local} if the precomposition map
\begin{equation*}
  f_i^\ast : (B_i \to X) \lra (A_i \to X)
\end{equation*}
is an equivalence for each $i:I$. For any family of maps $f$ in $\UU$, the subuniverse of $f$-local types is a reflective subuniverse. Moreover, if each $B_i$ is contractible, then this reflective subuniverse is a modality. In this case, the modal operator $X\mapsto LX$ is called \define{nullification}.
When $L$ is accessible, $L'$ is also accessible and is presented by the
family $\{ \susp f_i \}$ of suspensions of the maps presenting $L$.
See~\cite{RSS} and~\cite{CORS} for these results.

\subsection{Main result}

We collect together many of the results of this paper into one theorem.%
\footnote{We repeat this theorem on the last two pages for quick reference
and convenient printing.}
After this, we show how the large diagram looks in the case of the
$n$-truncation modality.

\ifthenelse{\boolean{end}}{\thispagestyle{empty}}{}
\renewcommand{\themainthm}{A}
\begin{mainthm}\label{thm:main}
  Let $L$ be a reflective subuniverse of $\UU$.
  Consider the following diagram:
  
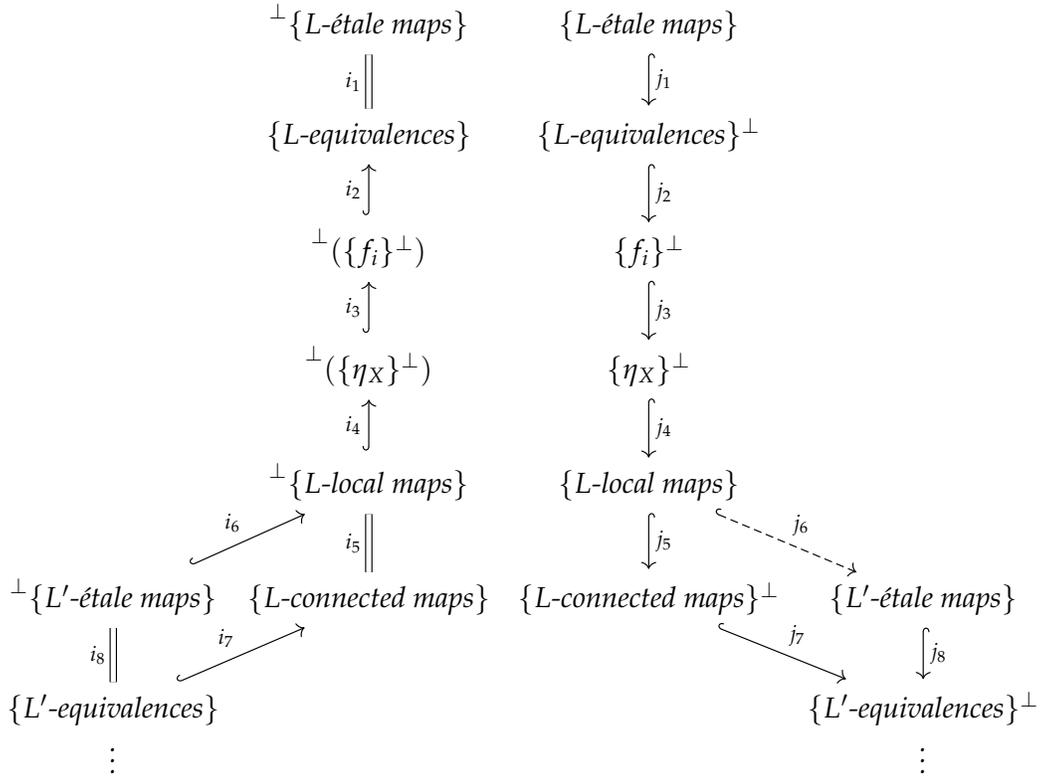
\begin{equation*}
    \begin{tikzcd}[column sep=0pt]
      & \prebot{\{L\text{-\'etale maps}\}} & \{L\text{-\'etale maps}\} \arrow[d,hook,"j_1"] \\
      & \{L\text{-equivalences}\} \arrow[u,equals,"i_1"] & \{L\text{-equivalences}\}^\bot \arrow[d,hook,"j_2"] \\
      & \prebot{(\{f_i\}^\bot)} \arrow[u,hook,"i_2"] & \{f_i\}^\bot \arrow[d,hook,"j_3"] \\
      & \prebot{(\{\eta_X\}^\bot)} \arrow[u,hook,"i_3"] & \{\eta_X\}^\bot \arrow[d,hook,"j_4"] \\
      & \prebot{\{L\text{-local maps}\}} \arrow[u,hook,"i_4"] & \{L\text{-local maps}\} \arrow[d,hook,"j_5"] \arrow[dr,hook,dashed,"j_6"]\\
      \prebot\{L'\text{-\'etale maps}\} \arrow[ur,hook,"i_6"]
      & \{L\text{-connected maps}\} \arrow[u,equals,"i_5"] & \{L\text{-connected maps}\}^\bot  \arrow[dr,hook,"j_7"]
      & \{L'\text{-\'etale maps}\} \arrow[d,hook,"j_8"] \\
      \{L'\text{-equivalences}\} \arrow[u,equals,"i_8"] \arrow[ur,hook,"i_7"]
      & & & \{L'\text{-equivalences}\}^\bot \ar[d, draw=none, "\raisebox{+1.5ex}{\vdots}" description] \\
      {} \ar[u, draw=none, "\raisebox{+1.5ex}{\vdots}" description] & & & {}
    \end{tikzcd}
    \vspace*{-11pt} %
  \end{equation*}
  The arrows should be interpreted as implications
  between propositions.  For example, $j_1$ means that
  for any reflective subuniverse $L$, if $f$ is $L$-\'etale, then it is
  right-orthogonal to every $L$-equivalence.
  The row concerning the family of maps $\{f_i\}$ only applies if the reflective
  subuniverse is presented by $\{f_i\}$, but the composites $i_2 \circ i_3$
  and $j_3 \circ j_2$ apply in general.
  The diagram continues infinitely downwards, with $L$ replaced by $L'$,
  the family $\{ f_i \}$ replaced by the family $\{ \susp f_i \}$ of suspensions,
  and the subscripts all increased by $7$ in each subsequent column.

  \ifthenelse{\boolean{end}}{\newpage}{}
  \ifthenelse{\boolean{end}}{\thispagestyle{empty}}{}
  \setlist{itemsep=2pt} %
  \begin{enumerate}
  \setlist{itemsep=0pt} %
  \item\label{it:prebot}
  The equalities $i_{1+7n}$ and $i_{5+7n}$ hold. Consequently,
  each class on the left is obtained from its mirror image class on the right by
  applying $\prebot{(-)}$.
  (However, it is not always the case that the class on the right is obtained
  by applying $(-)^{\bot}$ to the class on the left.)
  \item\label{it:inclusions}
  All of the inclusions, except possibly the dashed inclusions $j_{6+7n}$, hold.
  \pagebreak[2] %
  \item\label{it:L'-modality} The following are equivalent:
    \begin{enumerate}
    \item $L'$ is a modality.
    \item The inclusion $j_8$ is an equality.
    \item The inclusion $j_6$ exists.
    \end{enumerate}
  \item\label{it:L-modality} The following are equivalent:
    \begin{enumerate}
    \item $L$ is a modality.
    \item The inclusion $i_4$ is an equality.
    \item The inclusion $j_4$ is an equality.
    \item The inclusion $j_5$ is an equality.
    \end{enumerate}
    Also, when $L$ is a modality, $j_1$ is an equality and the
    conditions in \ref{it:L'-modality} hold.
    If, in addition, $L$ is presented as a nullification, then $i_3$ and $j_3$ are equalities.
  \item\label{it:lex} The following are equivalent:
    \begin{enumerate}
    \item $L$ is a lex modality.
    \item $L$ is a lex reflective subuniverse.
    \item All of the displayed vertical (non-diagonal) inclusions are equalities
          (but not necessarily $i_9$, $j_9$, etc.)
    \item The inclusion $j_4 \circ \cdots \circ j_1$ is an equality,
          i.e., every $L$-local map is $L$-\'etale.
    \item The inclusion $j_4 \circ j_3 \circ j_2$ is an equality.
    \item The inclusion $i_2 \circ i_3 \circ i_4$ is an equality,
          i.e., every $L$-equivalence is $L$-connected.
    \end{enumerate}
    Also, when $L$ is a lex modality, the conditions in \ref{it:L-modality} hold.
  \item\label{it:L=L'} The following are equivalent:
    \begin{enumerate}
    \item $L = L'$, i.e., the composite inclusion $\{ L\text{-local maps} \} \hookrightarrow \{ L'\text{-local maps} \}$ is an equality.
    \item All of the inclusions are equalities, including those not displayed.
    \item Every $L'$-\'etale map is $L$-local, i.e., there is an inclusion going
          in the opposite direction to~$j_6$.
    \item $L$ is \define{cotopological} (see \cref{se:L=L'}).
    \item $L$ is lex and every unit $\eta$ is surjective.
    \item $L$ is lex and every mere proposition is $L$-local.
    \end{enumerate}
    Also, when $L = L'$, the conditions in \ref{it:lex} hold.
  \item\label{it:strict}
    For each of the inclusions not drawn as equalities,
    including $j_6$ and those not displayed,
    there exists an accessible reflective subuniverse $L$ making the inclusion strict.
  \item\label{it:neither}
    In general, neither of $\{L\text{-connected maps}\}^\bot$ and
    $\{L'\text{-\'etale maps}\}$ includes in the other.
  \end{enumerate}
\end{mainthm}

Note that a composite inclusion, such as $j_4 \circ j_3$, is an equality if and
only if each factor is.
This lets us deduce claims involving classes of maps that might be more convenient.
For example, when $L$ is accessible, $j_4 \circ j_3$ being an equality implies that
$L$ is a modality, allowing us to give a condition in terms of the generators:
if every $L$-local map is right orthogonal to the generators, then $L$ is a modality.

The proof of \cref{thm:main} occupies the remainder of the paper,
with claim ($n$) of the theorem proved in Section $n$.
In some cases, we prove additional equivalent conditions that are
not stated in \cref{thm:main} to reduce its length.
For example, as an extension of \ref{it:L'-modality},
we give additional characterizations of when $L'$ is a modality
in \cref{thm:L'-is-a-modality}.
Item~\ref{it:L=L'} is related to work of Vergura in the context of
$\infty$-toposes~\cite[Section~7.2]{Vergura1}.
We explain this further in \cref{se:L=L'} and give additional equivalent conditions
in \cref{thm:L=L'}.

\medskip

Some natural open questions remain, of which we mention two.

\begin{introquestion}\label{q:j_1}
Does $j_1$ being an equality imply that $L$ is a modality?
The converse holds by~\ref{it:L-modality}, and the ``shifted'' version holds, namely
that $j_8$ being an equality implies that $L'$ is a modality,
by~\ref{it:L'-modality}.
We give a partial answer in \cref{prop:j_1}, where we show that 
if $j_1$ is an equality and the units are surjective, then $L$ is a modality.
\end{introquestion}

\begin{introquestion}
We know that the inclusion $j_6$ does not always exist.
Is there some class in the right-hand column that always includes in
$\{L'\text{-\'etale maps}\}$?
We thank Evan Cavallo for asking this question. %
\end{introquestion}

\begin{introeg}
  We illustrate here how the large diagram looks in the case of the
  $n$-truncation modality~\cite[Section~7.3]{hottbook}, for $n \geq -1$:
  \begin{equation*}
  \small
    \begin{tikzcd}[column sep=-2pt]
      & \prebot{\{n\text{-\'etale maps}\}} & \{n\text{-\'etale maps}\} \arrow[d,equals,"j_1"] \\
      & \{n\text{-equivalences}\} \arrow[u,equals,"i_1"] & \{n\text{-equivalences}\}^\bot \arrow[d,hook,"j_2"] \\
      & \prebot{(\{S^{n+1} \to 1\}^\bot)} \arrow[u,hook,"i_2"] & \{S^{n+1} \to 1\}^\bot \arrow[d,equals,"j_3"] \\
      & \prebot{(\{\eta_X\}^\bot)} \arrow[u,equals,"i_3"] & \{\eta_X\}^\bot \arrow[d,equal,"j_4"] \\
      & \prebot{\{n\text{-truncated maps}\}} \arrow[u,equals,"i_4"] & \{n\text{-truncated maps}\} \arrow[d,equals,"j_5"] \arrow[dr,hook,"j_6"]\\
      \prebot\{(n+1)\text{-\'etale maps}\} \arrow[ur,hook,"i_6"]
      & \{n\text{-connected maps}\} \arrow[u,equals,"i_5"] & \{n\text{-connected maps}\}^\bot  \arrow[dr,hook,"j_7"]
      & \{(n+1)\text{-\'etale maps}\} \arrow[d,equals,"j_8"] \\
      \{(n+1)\text{-equivalences}\} \arrow[u,equals,"i_8"] \arrow[ur,hook,"i_7"]
      & & & \{(n+1)\text{-equivalences}\}^\bot \ar[d, draw=none, "\raisebox{+1.5ex}{\vdots}" description] \\
      {} \ar[u, draw=none, "\raisebox{+1.5ex}{\vdots}" description] & & & {}
    \end{tikzcd}
    \vspace*{-9pt} 
  \end{equation*}
  The $n$-truncation modality is discussed in
  \cref{se:inclusions,se:L-modality,se:lex,se:strict}.
  The $n$-equivalences, $n$-connected maps and $n$-truncated maps are
  familiar, and by~\cite[Theorem~3.10]{CR}, the $n$-\'etale maps are
  precisely the maps that are right orthogonal to the inclusion
  $1 \to S^{n+1}$ of the basepoint.
\end{introeg}

In \cref{se:strict}, we also discuss the case of localization away
from a natural number $k$.

\setcounter{secnumdepth}{2}

\section{\texorpdfstring{$L$}{L}-connected maps and \texorpdfstring{$L$}{L}-equivalences}\label{se:prebot}

\setlist[enumerate,1]{label={(\roman*)}}
\setlist[enumerate,2]{label={(\alph*)}}

In this section we prove \cref{thm:main}\ref{it:prebot}.
We begin by proving the equalities $i_1$ and $i_5$, i.e., that
\begin{align*}
  \{L\text{-equivalences}\} & = \prebot{\{L\text{-\'etale maps}\}} \\
  \{L\text{-connected maps}\} & = \prebot{\{L\text{-local maps}\}}.
\end{align*}
The equalities $i_{1+7n}$ and $i_{5+7n}$ for $n > 0$ follow by replacing $L$
with $L^{(n)}$, where $(-)^{(n)}$ denotes applying $(-)'$ $n$ times.

We begin with $i_1$, and include some additional equivalent conditions.

\begin{prop}\label{prop:L-equivalence}
  Let $L$ be a reflective subuniverse and
  consider a map $f:A\to B$. Then the following are equivalent:
  \begin{enumerate}
  \item\label{Leq1} The map $f$ is an $L$-equivalence.
  \item\label{Leq2} For every $L$-local type $X$, the precomposition map
    \begin{equation*}
      \blank\circ f : (B \to X) \lra (A \to X)
    \end{equation*}
    is an equivalence.
  \item\label{Leq3} The map $f$ is left orthogonal to every map between $L$-local types.
  \item\label{Leq4} The map $f$ is left orthogonal to every $L$-\'etale map.
  \end{enumerate}
\end{prop}

\begin{proof}
  The first two statements are equivalent by \cite[Lemma~2.9]{CORS}.
  To see that \ref{Leq2} implies \ref{Leq3}, assume that $f$ satisfies \ref{Leq2}.
  Consider a map $g: X \to Y$ between $L$-local types.
  Then we have a commuting square
  \begin{equation*}
    \begin{tikzcd}
      X^B \arrow[r] \arrow[d] & Y^B \arrow[d] \\
      X^A \arrow[r] & Y^A .
    \end{tikzcd}
  \end{equation*}
  Since both vertical maps are equivalences by~\ref{Leq2},
  the square is a pullback square, and so $f$ is left orthogonal to $g$.

  To see that \ref{Leq3} implies \ref{Leq2}, consider an $L$-local type $X$. Then $f$ is left orthogonal to the map $X\to \unit$, which is just another way of saying that \ref{Leq2} holds.

  Note that \ref{Leq3} implies \ref{Leq4} since if $f$ is left orthogonal to any map between $L$-local types, then it is also left orthogonal to any base change of a map between $L$-local types. Furthermore, \ref{Leq4} implies \ref{Leq3} since any map between $L$-local types is clearly $L$-\'etale.
\end{proof}

Next we give the equality $i_5$, again with an extra condition.

\begin{prop}\label{prop:L-connected}
  Let $L$ be a reflective subuniverse and
  consider a map $f:A\to B$. Then the following are equivalent:
  \begin{enumerate}
  \item\label{Lconn1} The map $f$ is $L$-connected.
  \item\label{Lconn2} For each family $P$ of $L$-local types over $B$, the precomposition map
    \begin{equation*}
      \blank\circ f : \Big(\prd{y:B}P(y)\Big) \lra \Big(\prd{x:A}P(f(x))\Big)
    \end{equation*}
    is an equivalence.
  \item\label{Lconn3} The map $f$ is left orthogonal to every $L$-local map.
  \end{enumerate}
\end{prop}

\begin{proof}
  The equivalence of \ref{Lconn1} and \ref{Lconn2} is Lemma~1.36 in~\cite{RSS}.
  Therefore it suffices to show that \ref{Lconn2} and \ref{Lconn3} are equivalent.

  Condition~\ref{Lconn2} is equivalent to requiring that for each family $P$
  of $L$-local types over $B$ and each map $i$ making the square
  \[
    \begin{tikzcd}
      A \arrow[d,swap,"f"] \arrow[r,"i"] & \dsm{b:B} P(b) \arrow[d,"\pr_1"] \\
      B \arrow[r,swap,"\id"] & B
    \end{tikzcd}
  \]
  commute, the type of lifts in the square is contractible.
  \ref{Lconn3} is equivalent to requiring that for each square
  \[
    \begin{tikzcd}
      A \arrow[d,swap,"f"] \arrow[r,"i"] & \dsm{y:Y} P(y) \arrow[d,"\pr_1"] \\
      B \arrow[r,swap,"j"] & Y ,
    \end{tikzcd}
  \]
  where $Y$ is a type and $P$ is a family of $L$-local types over $Y$,
  the type of lifts is contractible.
  So \ref{Lconn3} clearly implies \ref{Lconn2}.
  In the second square, the pullback of the map $\pr_1$ along $j$
  is the map $\pr_1 : \sm{b:B} P(j(b)) \to B$.
  But since lifts in the second square correspond to lifts of $f$
  against this pullback, \ref{Lconn2} implies \ref{Lconn3}.
\end{proof}

We now prove the remaining claim in \cref{thm:main}\ref{it:prebot}, namely
that each class on the left-hand-side of the diagram is obtained from its mirror
image class on the right by applying $\prebot{(-)}$.
It is enough to prove this for all but the last row, since we can take $L$ to be $L^{(n)}$.
There are only two non-trivial cases.

\begin{prop}
We have equalities
\begin{align*}
  \{L\text{-equivalences}\} & = \prebot{(\{L\text{-equivalences}\}^{\bot})} \\
  \{L\text{-connected maps}\} & = \prebot{(\{L\text{-connected maps}\}^{\bot})}.
\end{align*}
\end{prop}

\begin{proof}
  This follows from a more general fact about orthogonality:
  if $\mathcal{P}$ and $\mathcal{Q}$ are classes of maps such that $\mathcal{P}=\prebot{\mathcal{Q}}$, then
  $\prebot{(\mathcal{P}^\bot)} = \prebot{((\prebot{Q})^\bot)} = \prebot{Q} = \mathcal{P}$.
  So the claim follows from the equalities $i_1$ and $i_5$.
\end{proof}

\section{Existence of the inclusions}\label{se:inclusions}

We show in this section that all the inclusions in \cref{thm:main} hold, i.e., we prove claim~\ref{it:inclusions} of \cref{thm:main}. Since each class on the left is the class of maps left orthogonal to every map in the corresponding class on the right, it suffices to construct only the inclusions $j_r$, as well as $i_6$, since $j_6$ does not exist in general. By \cref{prop:L-connected,prop:L-equivalence}, we already have the inclusions $j_1$, $j_5$, and $j_8$.
We proceed to prove the remaining inclusions.

In the following proposition, we construct the inclusion $j_3\circ j_2$ in the case of an arbitrary reflective subuniverse, and the inclusions $j_2$ and $j_3$ in the case of an accessible reflective subuniverse.
Despite the notation, in the non-accessible case, $j_3 \circ j_2$ is not a composite.

\begin{prop}\label{prop:j2}
  Let $L$ be a reflective subuniverse.
  Then there is an inclusion
  \[ j_3 \circ j_2 : \{L\text{-equivalences}\}^{\bot} \hooklongrightarrow \{\eta_X\}^{\bot}.\]
  Moreover, if $L$ is an accessible reflective subuniverse presented by a family of maps $\{f_i\}$,
 then there are inclusions
 $j_2 : \{L\text{-equivalences}\}^{\bot} \hookrightarrow \{f_i\}^\bot$
 and $j_3 : \{f_i\}^{\bot} \hookrightarrow \{\eta_X\}^{\bot}$.
\end{prop}

\begin{proof}
  By~\cite[Lemma~2.9]{CORS}, every unit $\eta_X : X \to LX$ is an $L$-equivalence,
  from which the inclusion $j_3 \circ j_2$ follows.

  Now assume that $L$ is presented by the family $\{f_{i}\}$.
  Then~\cite[Lemma~2.9]{CORS} shows that each $f_i$ is an $L$-equivalence,
  from which the inclusion $j_2$ follows.

  Next, we prove the inclusion $j_3$.
  By~\cite[Theorem~2.41]{RSS}, the maps $\{f_i\}$ generate an orthogonal factorization system
  with left class $\prebot{(\{f_i\}^\bot)}$.
  By~\cite[Theorem~3.12]{CORS}, if the maps $\{f_i\}$ are in the left class of any
  orthogonal factorization system, then so are the unit maps $\eta_X$.
  Therefore, the units $\eta_X$ are in $\prebot{(\{f_i\}^\bot)}$,
  which is equivalent to saying that the inclusion $j_3$ exists.
\end{proof}

By the following proposition we obtain the inclusion $j_4$. 

\begin{prop}\label{prop:right-orthogonal-is-local}
  Let $L$ be a reflective subuniverse.
  A map $g:X\to Y$ that is right orthogonal to every unit $\eta:A\to LA$ is $L$-local.
\end{prop}

\begin{proof}
  Consider a map $g:X\to Y$ that is right orthogonal to every unit $\eta:A\to LA$, let $y:Y$, and write $F\defeq \fib_g(y)$. Our goal is to show that $F$ is $L$-local. Since the map $F\to \unit$ is a base change of the map $g:X\to Y$, it follows that the map $F\to\unit$ is right orthogonal to every unit $\eta:A\to LA$. In other words, the precomposition map
  \begin{equation*}
    \eta^\ast : (LA \to F) \lra (A \to F)
  \end{equation*}
  is an equivalence, for every type $A$. If we take $A$ to be $F$, then it easily follows that $F$ is a retract of $LF$, and hence that $F$ is $L$-local.
\end{proof}

As a corollary of having constructed the inclusions $i_2$, $i_3$, and $i_4$, we obtain the well-known fact that any $L$-connected map is an $L$-equivalence \cite[Lemma 1.35]{RSS}. The fact that every $L'$-equivalence is an $L$-connected map is shown in~\cite[Proposition 2.30]{CORS}, so we obtain the inclusions $i_7$ and $j_7\jdeq i_7^\bot$, and it follows that we obtain the tower of inclusions
\begin{equation*}
  \cdots\subseteq\{\text{$L'$-equivalences}\}\subseteq\{\text{$L$-connected maps}\}\subseteq\{\text{$L$-equivalences}\}.
\end{equation*}
We note that, in the case where $L$ is $n$-truncation for $n\geq-1$, these inclusions are all strict.

To conclude this section, we note that the inclusion $i_6$ is obtained from the inclusion $i_7$ by the fact that $i_5$ and $i_8$ are equalities.

\section{Characterizing when \texorpdfstring{$L'$}{L'} is a modality}
\label{se:L'-modality}

In this section, we give a variety of characterizations of when $L'$ is a modality,
which in particular proves \cref{thm:main}\ref{it:L'-modality}.

Before we give the characterization, we recall that an \define{$L$-cartesian square} is a commuting square
\begin{equation*}
  \begin{tikzcd}
    E' \arrow[d,swap,"{p'}"] \arrow[r,"g"] & E \arrow[d,"p"] \\
    B' \arrow[r,swap,"f"] & B
  \end{tikzcd}
\end{equation*}
for which the \define{gap map} $E'\to B'\times_B E$ is $L$-connected.

\begin{thm}\label{thm:L'-is-a-modality}
  Let $L$ be a reflective subuniverse. Then the following are equivalent:
  \begin{enumerate}
  \item \label{item:L'-is-modality}The reflective subuniverse $L'$ is a modality.
  \item \label{item:is-L'-etale-is-right-orth-L'-equiv}Every map that is right orthogonal to every $L'$-equivalence is $L'$-\'etale.
    That is, $j_8$ is an equality.
  \item \label{item:is-L'-etale-is-L-local}Every $L$-local map is $L'$-\'etale.  That is, the inclusion $j_6$ exists.
  \item \label{item:stable-L'-equivalences}The $L'$-equivalences are stable under base change by $L$-local maps, i.e., for any pullback square
    \begin{equation*}
      \begin{tikzcd}
        E' \arrow[d,swap,"{p'}"] \arrow[r,"g"] & E \arrow[d,"p"] \\
        B' \arrow[r,swap,"f"] & B
      \end{tikzcd}
    \end{equation*}
    in which $f:B'\to B$ is an $L'$-equivalence and $p:E\to B$ is an $L$-local map, the map $g:E'\to E$ is an $L'$-equivalence.
  \item \label{item:L'-localization-Sigma}For every type $A$, and every family $P$ of $L$-local types over $L'A$, the natural map
    \begin{equation*}
      \Big(\sm{x:A}P(\eta'(x))\Big) \llra{\alpha} \Big(\sm{x':L'A}P(x')\Big)
    \end{equation*}
    is an $L'$-localization.
  \item \label{item:eta-path-is-connected}For every type $A$ and every $x,y:A$, the map
    \begin{equation*}
      \eta:(x=y) \lra L(x=y)
    \end{equation*}
    is $L$-connected.
  \item \label{item:L'-preserves-L-cartesian}The operation $L'$ preserves $L$-cartesian squares.
  \end{enumerate}
  In particular, if $L$ is a modality, then all of these statements hold.
\end{thm}

While most of the conditions involve $L$, conditions~\ref{item:L'-is-modality}
and~\ref{item:is-L'-etale-is-right-orth-L'-equiv} involve only $L'$.
This motivates \cref{q:j_1} in the Introduction and \cref{prop:j_1}.

\begin{proof}
  We will first prove the implications \ref{item:L'-is-modality} $\Longrightarrow$ \ref{item:is-L'-etale-is-right-orth-L'-equiv} $\Longrightarrow$ \ref{item:is-L'-etale-is-L-local} $\Longrightarrow$ \ref{item:stable-L'-equivalences}
  $\Longrightarrow$ \ref{item:L'-localization-Sigma} $\Longrightarrow$ \ref{item:L'-is-modality}. Afterwards, we will show that \ref{item:L'-is-modality} $\Longleftrightarrow$ \ref{item:eta-path-is-connected} and that \ref{item:eta-path-is-connected} $\Longleftrightarrow$ \ref{item:L'-preserves-L-cartesian}.

  To start with, we show that \ref{item:L'-is-modality} implies \ref{item:is-L'-etale-is-right-orth-L'-equiv}.
  If $L'$ is a modality, then by~\cite[Theorem~7.2]{CR}, summarized in the introduction, the pair
  \begin{equation*}
    (L'\text{-equivalences},L'\text{-\'etale maps})
  \end{equation*}
  is an orthogonal factorization system, so in particular~\ref{item:is-L'-etale-is-right-orth-L'-equiv} holds.

  The fact that \ref{item:is-L'-etale-is-right-orth-L'-equiv} implies \ref{item:is-L'-etale-is-L-local}
  is clear from the diagram in \cref{thm:main}, since if $j_8$ is an equality,
  then the inclusion $j_6$ is the composite $j_7\circ j_5$.

    Now we show that \ref{item:is-L'-etale-is-L-local} implies \ref{item:stable-L'-equivalences}.  Consider a pullback square
  \begin{equation*}
    \begin{tikzcd}
      E' \arrow[d,swap,"{p'}"] \arrow[r,"g"] & E \arrow[d,"p"] \\
      B' \arrow[r,swap,"f"] & B
    \end{tikzcd}
  \end{equation*}
  in which $p$ (and hence $p'$) are $L$-local, and $f$ is an $L'$-equivalence. This square fits in a commuting cube of the form
  \begin{equation*}
    \begin{tikzcd}
      & E' \arrow[dl] \arrow[d] \arrow[dr] \\
      L'E' \arrow[d] & B' \arrow[dl] \arrow[dr] & E \arrow[d] \arrow[dl,crossing over] \\
      L'B' \arrow[dr] & L'E \arrow[from=ul,crossing over] \arrow[d] & B \arrow[dl] \\
      & L'B . & \phantom{L'B'}
    \end{tikzcd}
  \end{equation*}
  Since the maps $p$ and $p'$ are assumed to be $L$-local, it follows by assumption that they are $L'$-\'etale. Therefore the back-left and the front-right squares are pullback squares. The back-right square is a pullback square by assumption. Since the unit map $\eta':B'\to L'B'$ is surjective, it therefore follows that the front-left square is also a pullback. However, the bottom map in this square is an equivalence, so it follows that the top map in this square is an equivalence. In other words: the map $g':E'\to E$ is an $L'$-equivalence.

  To see that \ref{item:stable-L'-equivalences} implies \ref{item:L'-localization-Sigma}, let $P$ be a family of $L$-local types over $L'A$ and consider
  the pullback square
  \begin{equation*}
    \begin{tikzcd}
      \dsm{x:A}P(\eta'(x)) \arrow[d,swap,"\pr_1"] \arrow[r,"\alpha"] & \dsm{x':L'A}P(x') \arrow[d] \\
      A \arrow[r,swap,"{\eta'}"] & L'A.
    \end{tikzcd}
  \end{equation*}
  By~\cite[Lemma~2.21]{CORS}, the codomain of $\alpha$ is $L'$-local, so it
  suffices to show that it is an $L'$-equivalence.
  In this square, the two vertical maps are $L$-local, while the map $\eta':A\to L'A$ is an $L'$-equivalence, so the claim follows from~\ref{item:stable-L'-equivalences}.
  
  To see that \ref{item:L'-localization-Sigma} implies \ref{item:L'-is-modality}, note that a special case of \ref{item:L'-localization-Sigma} is that the map
  \begin{equation*}
    \Big(\sm{x:A}\eta'(x)=a'\Big) \lra \Big(\sm{x':L'A}x'=a'\Big)
  \end{equation*}
  is an $L'$-localization, for any $a':L'A$.
  The codomain is contractible and the domain is the fibre of $\eta' : A \to L'A$,
  so $\eta'$ is $L'$-connected.  Thus, $L'$ is a modality.

  Now we show that \ref{item:L'-is-modality} and \ref{item:eta-path-is-connected} are equivalent. Note that $L'$ is a modality if and only if every unit map $\eta':A\to L'A$ is $L'$-connected.
  By~\cite[Remark 2.35]{CORS}, this happens if and only if
  \begin{equation*}
    \ap_{\eta'}:(x=y) \lra (\eta'(x)=\eta'(y))
  \end{equation*}
  is $L$-connected, for every $x, y : A$. Since $\ap_{\eta'}$ is an $L$-localization by \cite[Proposition~2.26]{CORS}, we see that the last condition holds if and only if $\eta:(x=y)\to L(x=y)$ is $L$-connected. This proves that \ref{item:L'-is-modality} is equivalent to \ref{item:eta-path-is-connected}.

  Finally, we show that \ref{item:eta-path-is-connected} and \ref{item:L'-preserves-L-cartesian} are equivalent. To show that \ref{item:eta-path-is-connected} implies \ref{item:L'-preserves-L-cartesian}, consider an $L$-cartesian square and its $L'$-localization
  \begin{equation*}
    \begin{tikzcd}
      C \arrow[r,"q"] \arrow[d,swap,"p"] & B \arrow[d,"g"] &[4em] L'C \arrow[r,"{L'q}"] \arrow[d,swap,"{L'p}"] & L'B \arrow[d,"{L'g}"] \\
      A \arrow[r,swap,"f"] & X & L'A \arrow[r,swap,"{L'f}"] & L'X.
    \end{tikzcd}
  \end{equation*}
  The gap map $L'C\to L'A\times_{L'X}L'B$ fits in a commuting square
  \begin{equation*}
    \begin{tikzcd}
      C \arrow[r] \arrow[d,swap,"{\eta'}"] & A\times_X B \arrow[d] \\
      L'C \arrow[r] & L'A\times_{L'X}\times L'B .
    \end{tikzcd}
  \end{equation*}
  In this square, the top map is $L$-connected by hypothesis and the map on the
  left is $L$-connected because $L'$ units are always $L$-connected.
  It follows from~\cite[Lemma~1.33]{RSS} that the bottom map is $L$-connected if and only if the vertical map on the right is $L$-connected. We proceed by showing that the map on the right is $L$-connected. This map is the map on total spaces
  \begin{equation*}
    \Big(\sm{x:A}{y:B}f(x)=g(y)\Big) \lra \Big(\sm{x':L'A}{y':L'B}(L'f)(x')=(L'g)(y')\Big)
  \end{equation*}
  induced by the $L$-connected maps $\eta':A\to L'A$ and $\eta':B\to L'B$, and the composite
  \begin{equation*}
    \begin{tikzcd}
      \big(f(x)=g(y)\big) \arrow[r,"\ap_{\eta'}"] & \big(\eta'(f(x))=\eta'(g(y))\big) \arrow[r,"\simeq"] & \big((L'f)(\eta'(x))=(L'g)(\eta'(y))\big)
    \end{tikzcd}
  \end{equation*}
  of $\ap_{\eta'}$ with an equivalence.
  By~\cite[Proposition~2.26]{CORS}, $\ap_{\eta'}$ is an $L$-localization, so
  this composite is also $L$-connected, as we have assumed \ref{item:eta-path-is-connected}.
  It follows from~\cite[Lemma~1.39]{RSS} that the map on total spaces is $L$-connected.
  This proves that \ref{item:eta-path-is-connected} implies \ref{item:L'-preserves-L-cartesian}.

  To see that \ref{item:L'-preserves-L-cartesian} implies \ref{item:eta-path-is-connected}, note that any pullback square is $L$-cartesian. Therefore it follows that the square
  \begin{equation*}
    \begin{tikzcd}
      L'(x=y) \arrow[r] \arrow[d] & \unit \arrow[d,"{\eta'(y)}"] \\
      \unit \arrow[r,swap,"{\eta'(x)}"] & L'A
    \end{tikzcd}
  \end{equation*}
  is $L$-cartesian. Equivalently, the map $L'(x=y)\to L(x=y)$ is $L$-connected. Since we have a commuting triangle
  \begin{equation*}
    \begin{tikzcd}[column sep=0]
      & (x=y) \arrow[dl] \arrow[dr] \\
      L'(x=y) \arrow[rr] & & L(x=y),
    \end{tikzcd}
  \end{equation*}
  we see that $\eta:(x=y)\to L(x=y)$ is a composite of $L$-connected maps. Therefore it is $L$-connected, which proves that \ref{item:L'-preserves-L-cartesian} implies \ref{item:eta-path-is-connected}.
\end{proof}

\begin{rmk}
  Many of the properties in \cref{thm:L'-is-a-modality} are in fact only slightly
  stronger than properties that hold for an arbitrary reflective subuniverse.
  For example, property \ref{item:L'-localization-Sigma} says that for every
  family $P$ of $L$-local types over a type the form $L'A$, the natural map
  \begin{equation*}
    \Big(\sm{x:A}P(\eta'(x))\Big) \llra{\alpha} \Big(\sm{x':L'A}P(x')\Big)
  \end{equation*}
  is an $L'$-equivalence.
  In fact, for every reflective subuniverse $L$, this map is $L$-connected.
  To show this, it suffices to verify condition \ref{Lconn2} of \cref{prop:L-connected},
  namely that the precomposition map
  \begin{equation*}
    \Big(\prd{z:\sm{x':L'A}P(x')}Q(z)\Big) \lra \Big(\prd{z:\sm{x:A}P(\eta'(x))}Q(\alpha(z)\Big)
  \end{equation*}
  is an equivalence for every family $Q$ of $L$-local types over the total space of $P$.
  This last equivalence follows from~\cite[Proposition 2.22]{CORS}.
\end{rmk}

Finally, we also note that if $L'$ is a modality, then the converse of property \ref{item:stable-L'-equivalences} also holds. We will use the following lemma to prove this fact.

\begin{lem}\label{lem:units-surj}
  Suppose that $L$ is a reflective subuniverse for which all unit maps $\eta:X\to LX$ are surjective, and consider a commuting square
  \begin{equation*}
    \begin{tikzcd}
      E' \arrow[d,swap,"{p'}"] \arrow[r,"g"] & E \arrow[d,"p"] \\
      B' \arrow[r,swap,"f"] & B
    \end{tikzcd}
  \end{equation*}
  in which $p$ and $p'$ are $L$-\'etale and $f$ is an $L$-equivalence. Then the following are equivalent:
  \begin{enumerate}
  \item The square is a pullback square.
  \item The map $g$ is an $L$-equivalence.
  \end{enumerate}
\end{lem}

\begin{proof}
    The commuting square in the claim fits in a commuting cube of the form
  \begin{equation*}
    \begin{tikzcd}
      & E' \arrow[dl] \arrow[d] \arrow[dr] \\
      LE' \arrow[d] & B' \arrow[dl] \arrow[dr] & E \arrow[d] \arrow[dl,crossing over] \\
      LB' \arrow[dr] & LE \arrow[from=ul,crossing over] \arrow[d] & B \arrow[dl] \\
      & LB . & \phantom{LB'}
    \end{tikzcd}
  \end{equation*}
  Since $p$ and $p'$ are assumed to be $L$-\'etale, it follows that the back-left and front-right squares are pullback squares. Since $f$ is assumed to be an $L$-equivalence,
the front-left square is a pullback square if and only if $g$ is an $L$-equivalence.
Thus it suffices to show that the front-left square is a pullback square if and only if the back-right square is a pullback square. We indeed have this equivalence, since the unit map $B'\to LB'$ is surjective.
\end{proof}

\begin{prop}
  Suppose $L$ is a reflective subuniverse such that $L'$ is a modality. Then any commuting square
  \begin{equation*}
    \begin{tikzcd}
      E' \arrow[d,swap,"{p'}"] \arrow[r,"g"] & E \arrow[d,"p"] \\
      B' \arrow[r,swap,"f"] & B
    \end{tikzcd}
  \end{equation*}
  in which $p$ and $p'$ are $L$-local and $f$ and $g$ are $L'$-equivalences is a pullback square.
\end{prop}

\begin{proof}
  We apply the previous lemma with the reflective subuniverse $L'$. Indeed, all of the unit maps $\eta'$ are surjective. Moreover, since $L'$ is assumed to be a modality, it follows that every $L$-local map is $L'$-\'etale.
\end{proof}

\section{Characterizing modalities}\label{se:L-modality}

In this section, we prove \cref{thm:main}\ref{it:L-modality} along with further conditions that are equivalent to $L$ being a modality.

\begin{thm}\label{thm:L-modality}
  Let $L$ be a reflective subuniverse.
  Then the following are equivalent:
  \begin{enumerate}
  \item\label{Lmod1} $L$ is a modality.
  \item\label{Lmod2} The inclusion $i_4 : \prebot\{ L\text{-local maps}\} \hookrightarrow \prebot(\{\eta_X\}^{\bot})$ is an equality.
  \item\label{Lmod3} The inclusion $j_4 : \{\eta_X\}^{\bot} \hookrightarrow \{L\text{-local maps}\}$ is an equality.
  \item\label{Lmod4} The inclusion $j_5 : \{L\text{-local maps}\} \hookrightarrow \{L\text{-connected maps}\}^{\bot}$ is an equality.
  \item\label{Lmod5} The inclusion $j_5 \circ j_4:  \{\eta_X\}^{\bot} \hookrightarrow \{L\text{-connected maps}\}^{\bot}$ is an equality.
  \item\label{Lmod6} $L$ preserves pullback squares
    \begin{equation*}
      \begin{tikzcd}
        E' \arrow[d] \arrow[r] & E \arrow[d,"p"] \\
        B' \arrow[r] & B
      \end{tikzcd}
    \end{equation*}
    in which the map $p$ is $L$-\'etale.
  \item\label{Lmod7}The $L$-equivalences are stable under base change by $L$-\'etale maps.
  \end{enumerate}
  Also, when $L$ is a modality,
  $j_1 : \{L\text{-\'etale maps}\} \hookrightarrow \{L\text{-equivalences}\}^{\bot}$
  is an equality and the conditions in \cref{thm:main}\ref{it:L'-modality} hold.
  If, in addition, $L$ is presented as a nullification at the family $\{f_i\}$,
  then $i_3 : \prebot(\{\eta_X\}^{\bot}) \hookrightarrow \prebot(\{f_i\}^{\bot})$
  and $j_3 : \{f_i\}^{\bot} \hookrightarrow \{\eta_X\}^{\bot}$ are equalities.
\end{thm}

\begin{proof}
  We first show that \ref{Lmod1} $\Longrightarrow$ \ref{Lmod5} $\Longrightarrow$ \ref{Lmod4} $\Longrightarrow$ \ref{Lmod1}.
  When $L$ is a modality, $j_5 \circ j_4$ is an equality, since for a modality,
  the units $\eta_X$ are $L$-connected. %
  Having $j_5 \circ j_4$ an equality clearly implies that $j_5$ is an equality.
  We close the circle by showing that when $j_5$ is an equality, $L$ is a modality.
  This is because $\{L\text{-connected maps}\}^\bot$ is closed under
  composition, and $\{L\text{-local maps}\}$ being closed under composition
  implies that the $L$-local types are $\Sigma$-closed, which we are taking
  as our definition of a modality.

  Now we show that \ref{Lmod3} $\Longrightarrow$ \ref{Lmod2} $\Longrightarrow$ \ref{Lmod5} $\Longrightarrow$ \ref{Lmod3}. The first implication follows from applying $\prebot(-)$,
  the second follows from applying $(-)^{\bot}$,
  and the third is trivial.

  Next we show that \ref{Lmod1} $\Longrightarrow$ \ref{Lmod6} $\Longrightarrow$ \ref{Lmod7} $\Longrightarrow$ \ref{Lmod1}. The implication \ref{Lmod1} $\Longrightarrow$ \ref{Lmod6} is \cite[Corollary 5.2]{CR}. To show that \ref{Lmod6} $\Longrightarrow$ \ref{Lmod7}, consider a pullback square
    \begin{equation*}
      \begin{tikzcd}
        E' \arrow[d] \arrow[r,"g"] & E \arrow[d,"p"] \\
        B' \arrow[r,swap,"f"] & B
      \end{tikzcd}
    \end{equation*}
    in which $p$ is $L$-\'etale and $f$ is an $L$-equivalence. Since $L$ preserves this pullback square, it follows that $Lg$ is a base change of the equivalence $Lf$. Therefore $Lg$ is an equivalence, so $g$ is an $L$-equivalence.

  To show that \ref{Lmod7} $\Longrightarrow$ \ref{Lmod1}, simply consider the pullback square
  \begin{equation*}
    \begin{tikzcd}
      \fib_\eta(y) \arrow[d] \arrow[r] & X \arrow[d,"\eta"] \\
      \unit \arrow[r,swap,"y"] & LX .
    \end{tikzcd}
  \end{equation*}
  The map $\eta:X\to LX$ is of course an $L$-equivalence, and the bottom map is $L$-\'etale because it is a map between $L$-local objects. Therefore the map on the left is an $L$-equivalence. In other words, $\eta$ is an $L$-connected map and therefore $L$ is a modality.

  For the final claim, assume that $L$ is a modality.
  Then $j_1$ is an equality, by \cite[Theorem~7.2]{CR}, summarized
  in the Introduction.
  In addition, $L'$ is a modality, by~\cite[Remark~2.16]{CORS},
  so the conditions in \cref{thm:main}\ref{it:L'-modality} hold.

  If $L$ is presented as a nullification, then the next result
  implies that $j_3$ is an equality, and it then follows that
  $i_3 = \prebot{j_3}$ is an equality.
\end{proof}

We now prove the result that implies the last claim of \cref{thm:main}\ref{it:L-modality}.

\begin{prop}\label{prop:j34}
  Let $L$ be presented as nullification at a family of types $B_i$,
  so that each $f_i$ is the unique map $B_i \to 1$.
  Then $j_4 \circ j_3 : \{f_i\}^{\bot} \hookrightarrow \{L\text{-local maps}\}$
  is an equality.
\end{prop}

\begin{proof}
  Let $g : X \to Y$ be an $L$-local map.
  We must show that the type of lifts in any commuting square
  \[
    \begin{tikzcd}
      B_i \arrow[d,swap,"f_i"] \arrow[r,"h"] & X \arrow[d,"g"] \\
      1 \arrow[r,swap,"k"] & Y
    \end{tikzcd}
  \]
  is contractible.  But the type of lifts in such a square is the
  same as the type of lifts in the square
  \[
    \begin{tikzcd}
      B_i \arrow[d,swap,"f_i"] \arrow[r,"h"] & F \arrow[d,"g'"] \\
      1 \arrow[r,swap,"\id"] & 1 ,
    \end{tikzcd}
  \]
  where $g'$ is the pullback of $g$ along $k$.
  Since $F$ is $\{f_i\}$-local, the type of such lifts is con\-trac\-ti\-ble.
\end{proof}

Recall that, by~\cite[Theorem~2.41]{RSS}, any family of maps $\{f_i\}$
generates an orthogonal factorization system.
The above result implies that when the maps are of the form $B_i \to 1$,
this orthogonal factorization depends only on the localization presented
by the family of maps.
The next example shows that this is not true for a general presentation
of a modality.

\begin{eg}\label{eg:OGS}
  For each pointed type $B$, nullification at $B$ is presented in the
  standard way by the map $B \to 1$, but can also be presented by the
  map $1 \to B$.
  Consider $n$-truncation, for some $n : \N$.
  \cref{prop:j34} shows that for the standard presentation using the map
  $f : \sphere{n+1} \to 1$, we have $\{f\}^{\bot} = \{n\text{-truncated maps}\}$.
  On the other hand, consider the presentation of $n$-truncation
  using the map $g : 1 \to \sphere{n+1}$ for a chosen basepoint of $\sphere{n+1}$.
  By \cite[Theorem~3.10]{CR}, $\{g\}^{\bot}$ is equal to the $n$-\'etale maps,
  i.e., the \'etale maps for $n$-truncation.
  It is easy to give explicit examples showing that not every $n$-truncated map
  is $n$-\'etale, and this also follows from \cref{thm:main}\ref{it:lex},
  since $n$-truncation is not lex.
  Thus the inclusion $\{g\}^{\bot} \hookrightarrow \{f\}^{\bot}$ is strict.
  In other words, the orthogonal factorization systems generated by $f$ and $g$
  are not equal.
  In particular, for the presentation using $g$, $j_3$ is a strict inclusion.
\end{eg}

Now we give a result that partially answers \cref{q:j_1}.

\begin{prop}\label{prop:j_1}
  Let $L$ be a reflective subuniverse for which the unit maps $\eta : X \to LX$
  are all surjective.
  Then the following are equivalent:
  \begin{enumerate}
  \item\label{j_1:L-modality} The reflective subuniverse $L$ is a modality.
  \item\label{j_1:j_1} Every map that is right orthogonal to every $L$-equivalence is $L$-\'etale.
    That is, $j_1$ is an equality.
  \item\label{j_1:L-etale-pullback} The $L$-\'etale maps are closed under pullback.
  \end{enumerate}
  In particular, when $L = K'$ for some reflective subuniverse $K$,
  these conditions are equivalent, since the units for $L$ are surjective.
\end{prop}

\begin{proof}
The implication \ref{j_1:L-modality} $\implies$ \ref{j_1:j_1} was proved in \cref{thm:L-modality}.

To prove that \ref{j_1:j_1} $\implies$ \ref{j_1:L-etale-pullback},
suppose that $p : E \to B$ is $L$-\'etale.  Since $p$ is right-orthogonal to
all $L$-equivalences, the same is true for any pullback of $p$.
Therefore, any pullback of $p$ is $L$-\'etale.

Finally, we prove that \ref{j_1:L-etale-pullback} $\implies$ \ref{j_1:L-modality}
by showing that \ref{j_1:L-etale-pullback} implies \cref{thm:L-modality}\ref{Lmod6}.
In other words, we will show that $L$ preserves any pullback square
    \begin{equation*}
      \begin{tikzcd}
        E' \arrow[d,swap,"p'"] \arrow[r,"g"] & E \arrow[d,"p"] \\
        B' \arrow[r,swap,"f"] & B
      \end{tikzcd}
    \end{equation*}
in which $p$ is $L$-\'etale.
Since $p$ is $L$-\'etale, it follows from the pasting lemma
(see, for example,~\cite[Corollary~2.1.17]{rijkethesis})
that the outer rectangle in
    \begin{equation*}
      \begin{tikzcd}
        E' \arrow[d,swap,"p'"] \arrow[r,"g"] & E \arrow[d,"p"] \arrow[r,"\eta"] & LE \arrow[d,"Lp"] \\
        B' \arrow[r,swap,"f"] & B \arrow[r,swap,"\eta"] & LB
      \end{tikzcd}
    \end{equation*}
is a pullback.
This outer rectangle is equal to the outer rectangle in
    \begin{equation*}
      \begin{tikzcd}
        E' \arrow[d,swap,"p'"] \arrow[r,"\eta"] & LE' \arrow[d,"Lp'"] \arrow[r,"Lg"] & LE \arrow[d,"Lp"] \\
        B' \arrow[r,swap,"\eta"] & LB' \arrow[r,swap,"Lf"] & LB .
      \end{tikzcd}
    \end{equation*}
The square on the left is a pullback, since we are assuming that the pullback $p'$
of $p$ is again $L$-\'etale.
Since the units are surjective, it follows from~\cite[Theorem~2.3]{CR}
that the square on the right is a pullback, as required.
\end{proof}

\section{Characterizing lex modalities}\label{se:lex}

In this section we prove \cref{thm:main}\ref{it:lex}, characterizing lex modalities. We recall that \cite[Theorem~3.1]{RSS} already contains a long list of conditions on a modality that are equivalent to it being lex. One difference between Theorem~3.1 of \cite{RSS} and our \cref{thm:main}\ref{it:lex} is that we only assume $L$ to be a reflective subuniverse, whereas the theorem in \cite{RSS} assumes $L$ to be a modality.

We also note that the $(-2)$-truncation, which is the trivial lex modality given by $X\mapsto \unit$, is an example of a lex modality $L$ for which $L'$ is not lex. Indeed, in this case $L'$ is propositional truncation, which is not lex because it doesn't preserve the pullback square
\begin{equation*}
  \begin{tikzcd}
    \varnothing \arrow[d] \arrow[r] & \unit \arrow[d,"1"] \\
    \unit \arrow[r,swap,"0"] & \mathbf{2} .\!
  \end{tikzcd}
\end{equation*}
However, since $L'$ is a modality whenever $L$ is a modality, it is in particular the case that $L'$ is a modality whenever $L$ is a lex modality.

\begin{thm}\label{thm:lex}
  Let $L$ be a reflective subuniverse. Then the following are equivalent:
  \begin{enumerate}
  \item \label{it-thm-lex:lex}$L$ is a lex modality.
  \item \label{it-thm-lex:lexRSU}$L$ is a lex reflective subuniverse.
  \item \label{it-thm-lex:all-eqs}Each of the inclusions $i_1 \circ \cdots \circ i_5$, $i_8$, $j_5 \circ \cdots \circ j_1$ and $j_8$ is an equality.
  \item \label{it-thm-lex:j4-j1}The inclusion $j_4\circ \cdots \circ j_1$ is an equality, i.e., every $L$-local map is $L$-\'etale.
  \item \label{it-thm-lex:j4-j2}The inclusion $j_4\circ j_3\circ j_2$ is an equality.
  \item \label{it-thm-lex:i2-i4}The inclusion $i_2\circ i_3\circ i_4$ is an equality, i.e., every $L$-equivalence is $L$-connected.
  \end{enumerate}
\end{thm}

\begin{proof}
  It is clear that \ref{it-thm-lex:lex} implies \ref{it-thm-lex:lexRSU}.
  To prove the converse, suppose that $L$ is a lex reflective subuniverse.
  We show that $L$ is $\Sigma$-closed, in fibration form.
  Let $f : E \to B$ be an $L$-local map over an $L$-local base $B$.
  We will show that $E$ is $L$-local.
  For each $b : B$, consider the fibre $\fib_f(b)$.
  Applying $L$ gives a commutative diagram
\[
  \begin{tikzcd}
    \fib_f(b) \arrow[d]      \arrow[r,"\eta","\sim"'] & L\fib_f(b) \arrow[d] \\
    E         \arrow[d,"f"'] \arrow[r,"\eta"]         & LE         \arrow[d,"Lf"] \\
    B                        \arrow[r,"\eta","\sim"'] & LB
  \end{tikzcd}
\]
  in which the top and bottom arrows are equivalences by assumption.
  Since $L$ is lex, the right-hand column is also a fibre sequence.
  Since this holds for every $b : B$, we conclude that $\eta : E \to LE$
  is a fibrewise equivalence, and hence an equivalence.

  We next prove that \ref{it-thm-lex:lex} implies \ref{it-thm-lex:all-eqs}.
  Suppose that $L$ is a lex modality. Since $L$ is a modality, it follows by \cref{thm:main}\ref{it:prebot}--\ref{it:L-modality} that $i_1,i_4,i_5,i_8,j_1,j_4,j_5$ and $j_8$ are equalities. Furthermore, since $L$ is a lex modality,
 \cite[Theorem~3.1~(xii)]{RSS} says that $i_2 \circ \cdots \circ i_5$ is an equality.
 It follows that $j_5 \circ \cdots \circ j_2$ is equality, by applying $(-)^{\bot}$.

  The facts that \ref{it-thm-lex:all-eqs} implies \ref{it-thm-lex:j4-j1}, that \ref{it-thm-lex:j4-j1} implies \ref{it-thm-lex:j4-j2}, and that \ref{it-thm-lex:j4-j2} implies \ref{it-thm-lex:i2-i4} are immediate.
  For the last one, we apply $\prebot(-)$.

  We conclude by showing that \ref{it-thm-lex:i2-i4} implies \ref{it-thm-lex:lex}.
  If \ref{it-thm-lex:i2-i4} holds, then it follows that $i_4$ is an equality, and hence that $L$ is a modality.
  Therefore, using that $i_5$ is always an equality, we can
  apply~\cite[Theorem~3.1~(xii)]{RSS} to conclude that $L$ is lex.
\end{proof}

Note that the proof shows that \cite[Theorem~3.1~(xii)]{RSS} characterizes
lex modalities among all reflective subuniverses.

\section{Characterizing cotopological modalities}\label{se:L=L'}

In this section, we explain the concept of a cotopological modality
and prove \cref{thm:main}\ref{it:L=L'}.
For this, we need some definitions.

\begin{defn}
A map $f : X \to Y$ is \define{$\infty$-connected} if it is $n$-connected for each $n \geq -2$.
\end{defn}

It is equivalent to require that $\|f\|_n$ be an equivalence for each $n$.

\begin{defn}
A reflective subuniverse $L$ is \define{cotopological} if it is a lex modality
such that every $L$-connected map is $\infty$-connected.
\end{defn}

\begin{rmk}
This is one of three equivalent characterizations given in \cite[Theorem~3.22]{RSS},
another being given in \cref{thm:L=L'}\ref{LL'4} below.
It also corresponds to the notion of ``quasi-cotopological
localization'' of $\infty$-toposes given in \cite[Section~7.2]{Vergura1}.
One of Vergura's conditions, given in his Proposition 4.11, is that every
$L$-local map is $L$-\'etale, and by \cref{thm:main}\ref{it:lex},
this is equivalent to $L$ being lex.
Vergura's other condition is that every $L$-equivalence is $\infty$-connected,
and so Vergura's definition is equivalent to \cref{thm:L=L'}\ref{LL'6}.
Lurie's notion of ``cotopological localization''~\cite[Section~6.5.2]{lurie}
differs only in that it
also assumes accessibility.
\end{rmk}

We also need the following lemma.

\begin{lem}\label{lem:props-L-local}
Let $L$ be a reflective subuniverse.
If the unit maps $\eta : X \to LX$ are all surjective,
then every mere proposition is $L$-local.
The converse holds if $L$ is a modality.
\end{lem}

\begin{proof}
Let $P$ be a mere proposition and consider $\eta : P \to LP$.
We have the following chain of equivalences:
\begin{align*}
         \issurj(\eta)
 &=      \prd{p : LP} \| \fib_{\eta}(p) \|_{-1} \\
 &=      \prd{p : LP} \| \sm{q:P} \eta(q) = p \|_{-1} \\
 &\simeq \prd{p : LP} \| P \|_{-1} \\
 &\simeq \prd{p : LP} P \\
 &=      (LP \to P).
\end{align*}
If we have a map $LP \to P$, then $P$ is a retract of $LP$ and is
therefore $L$-local.
So the unit for $P$ is surjective if and only if $P$ is $L$-local,
which in particular proves the first claim.

Now suppose that $L$ is a modality and that every mere proposition is $L$-local.
Then, for any type $X$, dependent elimination gives us an equivalence
\[
  \prd{x:LX} \| \fib_{\eta}(x) \|_{-1} \simeq \prd{x:X} \| \fib_{\eta}(\eta(x)) \|_{-1}.
\]
The right hand side is inhabited, so $\eta : X \to LX$ is surjective.
\end{proof}

We can now give the main theorem of this section,
which proves and slightly extends \cref{thm:main}\ref{it:L=L'}.

\begin{thm}\label{thm:L=L'}
Let $L$ be a reflective subuniverse.  Then the following are equivalent:
    \begin{enumerate}
    \item\label{LL'1} $L = L'$, i.e., every $L'$-local map is $L$-local.
    \item\label{LL'2} All of the inclusions in \cref{thm:main} are equalities,
          including those not displayed.
    \item\label{LL'3} Every $L'$-\'etale map is $L$-local.
    \item\label{LL'4} $L$ is lex and every $L$-connected mere proposition is contractible.
    \item\label{LL'5} $L$ is cotopological.
    \item\label{LL'6} $L$ is lex and every $L$-equivalence is $\infty$-connected.
    \item\label{LL'7} $L$ is lex and every unit $\eta$ is surjective.
    \item\label{LL'8} $L$ is lex and every mere proposition is $L$-local.
    \item\label{LL'9} $L$ is lex and every truncated type is $L$-local.
    \end{enumerate}
\end{thm}

\begin{proof}
  \ref{LL'1} $\implies$ \ref{LL'2}:
  Suppose $L = L'$.  Then $\{L\text{-equivalences}\}^{\bot} = \{L'\text{-equivalences}\}^{\bot}$,
  which implies that $L$ is lex and $j_7$ is an equality.
  It follows that all of the displayed inclusions are equalities.
  Since $L' = L''$, the next seven inclusions are equalities as well.
  We see inductively that all of the inclusions are equalities.

  \ref{LL'2} $\implies$ \ref{LL'3}:
  Clearly, if all of the inclusions are equalities, then $j_6$ exists and is
  an equality.

  \ref{LL'3} $\implies$ \ref{LL'1}:
  Now suppose that every $L'$-\'etale map is $L$-local.
  If $X$ is $L'$-local, then $X \to 1$ is $L'$-\'etale, since $L'1 = 1$
  for every reflective subuniverse, so $X \to 1$ is $L$-local.
  In other words, the $L$-local types agree with the $L'$-local types,
  which means that $L = L'$.

  \ref{LL'1} $\implies$ \ref{LL'4}:
  Suppose $L = L'$.  Since \ref{LL'1} $\implies$ \ref{LL'2}, we see that $L$ is lex,
  using \cref{thm:main}\ref{it:lex}.
  And since every mere proposition is $L'$-local,
  it follows that every mere proposition is $L$-local.
  In particular, an $L$-connected mere proposition is contractible.

  For the remaining items, recall from \cref{thm:lex} that a
  lex reflective subuniverse is automatically a modality.

  \ref{LL'4} $\iff$ \ref{LL'5} is part of \cite[Theorem~3.22]{RSS}.
  And for a lex modality, being an $L$-equivalence is equivalent to
  being $L$-connected, so \ref{LL'5} $\iff$ \ref{LL'6}.

  \ref{LL'6} $\implies$ \ref{LL'7}:
  This is clear because every unit is an $L$-equivalence and
  $\infty$-connected maps are $(-1)$-connected, i.e., surjective.

  \ref{LL'7} $\implies$ \ref{LL'1}:
  Assume that $L$ is lex and all units are surjective.
  We will show that every $L'$-local type is $L$-local.
  Suppose that $X$ is $L'$-local, and consider the commuting square
  \[
    \begin{tikzcd}
      X \arrow[r,"\eta"] \arrow[d,swap,"\Delta"] & LX \arrow[d,"\Delta"] \\
      X^2 \arrow[r,swap,"\eta \times \eta"] & (LX)^2 .
    \end{tikzcd}
  \]
  Localizations respect products, so the bottom arrow is equivalent to the unit.
  Since $X$ is $L'$-local, the left-hand map is $L$-local.
  Since $L$ is lex, $L$-local maps are $L$-\'etale, so the square is a pullback.
  It follows that the induced map on the fibres, which is equivalent to
  $\ap_{\eta} : x_0 = x_1 \to \eta x_0 = \eta x_1$, is an equivalence.
  In other words, $\eta$ is a monomorphism.
  Since we also are assuming that the units are surjective, $\eta$ is an equivalence.
  That is, $X$ is $L$-local, as required.

  \ref{LL'7} $\iff$ \ref{LL'8}:  This follows from \cref{lem:props-L-local}, since $L$ is a modality.

  \ref{LL'8} $\implies$ \ref{LL'9}:  We know from the above that \ref{LL'8} implies \ref{LL'1},
  so we may assume that both of these hold.
  We prove \ref{LL'9} inductively.
  By \ref{LL'8}, every $(-1)$-type is $L$-local.
  For $n \geq -1$, let $X$ be an $(n+1)$-type.
  The identity types of $X$ are $n$-types and are $L$-local by induction.
  Therefore, $X$ is $L'$-local.  So, by \ref{LL'1}, $X$ is $L$-local.

  \ref{LL'9} $\implies$ \ref{LL'8}:  This is trivial.
\end{proof}

The equivalence between \ref{LL'1} and \ref{LL'6} is proved for $\infty$-toposes in
\cite[Theorem~7.3]{Vergura1}, and our proof that \ref{LL'7} implies \ref{LL'1} is
adapted from Vergura's proof.

\begin{eg}
In the homotopy theory of topological spaces, there is a reflective subuniverse
$L$ such that $L$ is not cotopological, but $L'$ is cotopological.
Put another way, we have $L \neq L' = L''$.
The only ingredient one needs to construct $L$ is a non-trivial
\define{acyclic} space, that is, a non-contractible space $B$ whose suspension
$\susp B$ is contractible, and it is well-known in topology that these exist.
Let $L$ be nullification at such a space $B$.
Then $L$ is non-trivial, as
$B$ is not $L$-local but every loop space is $L$-local.
But $L'$ is nullification at $\susp B \simeq 1$, by \cite[Lemma~2.15]{CORS}, so every space is $L'$-local
and $L' = L''$.
We expect the construction of such a $B$ to be possible in
homotopy type theory.
\end{eg}

\section{The inclusions can be strict}\label{se:strict}

In this section, we give examples showing that all of the
inclusions in the main theorem can be strict.

\begin{eg}\label{eg:degk}
  For $k : \N$, we define the \define{degree $k$ map}
  $\degg(k) : \sphere{1} \to \sphere{1}$
  by sending $\base$ to $\base$ and $\lloop$ to $\lloop^k$.
  Let $L_{\degg(k)}$ be localization with respect to this map.
  Then $L_{\degg(k)}' = L_{\susp \degg(k)}$, where $\susp$ denotes the suspension.
  More generally, for $n : \N$, $L_{\degg(k)}^{(n)} = L_{\susp^n \degg(k)}$,
  where $(-)^{(n)}$ denotes applying $(-)'$ $n$ times.
  In~\cite[Example 4.8]{CORS} it is shown that $L_{\degg(k)}$ is not a modality
  for $k > 1$, and the example generalizes to $L_{\susp^n \degg(k)}$ as follows.
  Consider the map $K(\Z, n+1) \to K(\Q, n+1)$ induced by the inclusion $\Z \to \Q$.
  The fibres of this map are $L_{\susp^n \degg(k)}$-local and $K(\Q, n+1)$ is
  $L_{\susp^n \degg(k)}$-local, but $K(\Z, n+1)$ is not $L_{\susp^n \degg(k)}$-local,
  so $L_{\susp^n \degg(k)}$ is not $\Sigma$-closed.
  These claims can be checked directly, and also follow from the case proven
  in~\cite{CORS}.
\end{eg}

\begin{thm}\label{thm:strict}
  For each of the inclusions not drawn as equalities,
  including $j_6$ and those not displayed,
  there exists an accessible reflective subuniverse $L$ making the inclusion strict.
\end{thm}

\begin{proof}
  By \cref{thm:main}\ref{it:L'-modality}, for a reflective subuniverse $L$
  such that $L'$ is not a modality, $i_4$, $j_4$, $j_5$ and $j_8$ are strict.
  Taking $L$ to be such an $L'$ gives a case where $j_1$ is strict.
  By \cref{eg:degk}, there is an accessible reflective subuniverse $L$ so that each $L^{(n)}$
  is not a modality, which shows that $i_{4+7n}$, $j_{4+7n}$, $j_{5+7n}$ and $j_{8+7n}$
  are all strict.

  We have that $i_2 = \prebot{j_2}$ and $j_2 = i_2^{\bot}$, so we either
  have that both are equalities, or that both are strict.
  Let $L$ be a nullification which is not lex.
  For example, one can take $\{f_i\} = \{S^1 \to 1\}$, so that $L$ is $0$-truncation.
  By \cref{thm:main}\ref{it:L-modality}, $j_4 \circ j_3$ is an equality.
  Since $L$ is not lex, we must have that $j_2$ is strict, by \cref{thm:main}\ref{it:lex}.
  Therefore, $i_2$ is strict as well.
  By considering $n$-truncation, we see that $i_{2+7n}$ and $j_{2+7n}$ can be strict.
  As a concrete example, when $L$ is $0$-truncation,
  the degree $k$ map $\degg(k) : S^1 \to S^1$, for $k > 1$,
  is $L$-local (and therefore in $\{f_i\}^{\bot}$)
  but is not $L$-\'etale (and therefore is not in $\{L\text{-equivalences}\}^{\bot}$).
  We can also see that $i_2$ is strict:
  $\degg(k)$ is an $L$-equivalence but is not $L$-connected.

  We have that $i_3 = \prebot{j_3}$ and $j_3 = i_3^{\bot}$, so we either
  have that both are equalities, or that both are strict.
  We saw in \cref{eg:OGS} that $j_3$ is strict for a non-standard presentation of $n$-truncation.
  It follows that $i_{3+7n}$ and $j_{3+7n}$ can be strict.

  Next we handle the diagonal arrows.
  Clearly $i_6$ is an equality if and only if $i_7$ is.
  Moreover, $i_7 = \prebot{j_7}$ and $j_7 = i_7^{\bot}$, so $i_7$ is an
  equality if and only if $j_7$ is.
  Examples where $i_7$ is not an equality are easy to find, such
  as $n$-truncation for any $n$.
  Note that \cref{thm:main}\ref{it:L-modality} implies that $j_6$ exists in this case,
  since $n$-truncation is a modality.
  However, $j_6$ is also not an equality, since $i_6 = \prebot{j_6}$.
  In fact, using \cref{thm:main}\ref{it:L=L'}, $j_6$ exists but is not an equality
  whenever $L'$ is a modality but $L \neq L'$.
  It follows that $i_{6+7n}$, $i_{7+7n}$, $j_{6+7n}$ and $j_{7+7n}$ can be strict.
\end{proof}

\section{Two non-inclusions}\label{se:neither}

In this section, we prove \cref{thm:main}\ref{it:neither}, which we
restate here.

\begin{thm}\label{thm:neither}
  In general, neither of $\{L\text{-connected maps}\}^\bot$ and
  $\{L'\text{-\'etale maps}\}$ includes in the other.
\end{thm}

Note that this is in contrast to the fact that after applying
$\prebot(-)$ we do get an inclusion of
$\prebot\{L'\text{-\'etale maps}\}$ into
$\prebot(\{L\text{-connected maps}\}^\bot) = \{L\text{-connected maps}\}$,
via the equality $i_8$ and the inclusion $i_7$.

\begin{proof}
  If $\{L\text{-connected maps}\}^\bot$ includes in $\{L'\text{-\'etale maps}\}$,
  then $j_6$ exists, which only happens when $L'$ is a modality by
  \cref{thm:main}\ref{it:L-modality}.
  We saw in \cref{eg:degk} an example of a reflective subuniverse $L$ such
  that $L'$ is not a modality.
  
  If $\{L'\text{-\'etale maps}\}$ includes in $\{L\text{-connected maps}\}^\bot$,
  then by applying $\prebot{(-)}$ we deduce that $i_7$ is an equality,
  which can fail by \cref{thm:main}\ref{it:strict}.
\end{proof}

\printbibliography[heading=bibintoc]

\cleardoublepage %
\section*{For quick reference and printing}

\setlist[enumerate,1]{label={(\arabic*)}}
\setlist[enumerate,2]{label={(\roman*)}}

\setboolean{end}{true}
{\def\label#1{} %

} 

\end{document}